\documentclass[twoside, 12pt]{amsart}
\usepackage[T1]{fontenc}
\usepackage{amsfonts,amsmath,amsthm,amscd,amssymb,latexsym,amsbsy,caption}
\usepackage[noadjust]{cite}
\usepackage[a4paper, top=3cm, bottom=3cm, left=3cm, right=3cm]{geometry}

\usepackage{tikz}
\usetikzlibrary{positioning, calc, trees, decorations.markings, patterns, arrows, arrows.meta}

\usepackage{diagmac2}

\usepackage{bm} 

\long\expandafter\def\csname cor1\endcsname#1{{\color{black} #1}}
\newcommand{\cor}[2]{\csname cor#1\endcsname{#2}}

\newtheorem{theorem}{Theorem}[section]
\newtheorem{lemma}[theorem]{Lemma}
\newtheorem{corollary}[theorem]{Corollary}
\newtheorem{proposition}[theorem]{Proposition}

\newtheorem{conjecture}[theorem]{Conjecture}

\theoremstyle{definition}
\newtheorem{definition}[theorem]{Definition}

\theoremstyle{remark}
\newtheorem{example}[theorem]{Example}
\newtheorem{remark}[theorem]{Remark}

\numberwithin{equation}{section}
\allowdisplaybreaks[4]

\makeatletter
\renewcommand{\subsection}{\@startsection{subsection}{2}%
  {0pt}{.7\linespacing\@plus\linespacing}{.5\linespacing}%
  {\normalfont\bfseries}}

\renewcommand{\p@enumii}{}

\@namedef{subjclassname@2020}{%
  \textup{2020} Mathematics Subject Classification}
\makeatother

\newcommand{\levdist}{1.2cm}

\DeclareMathOperator{\card}{card}

\def\<#1>{\langle #1 \rangle}

\newcommand{\acr}{\newline\indent}

\author{Oleksiy Dovgoshey}
\address{\textbf{Oleksiy Dovgoshey}\acr
Institute of Applied Mathematics and Mechanics
of the NAS of Ukraine, \acr Sloviansk, Ukraine \newline\indent
and Department of Mathematics and Statistics, University of Turku,  Finland.}
\email{oleksiy.dovgoshey@gmail.com, oleksiy.dovgoshey@utu.fi}

\author{Valentino Vito}
\address{\textbf{Valentino Vito}\acr
Faculty of Computer Science, Universitas Indonesia, Depok, Indonesia}
\email{valentino.vito11@ui.ac.id}

\title[Labeled rays]{Totally bounded ultrametric spaces \\generated by labeled rays}



\begin{document}

\begin{abstract}
We will say that an infinite tree $T$ is almost a ray if $T$ is the union of a ray and a finite tree. Let $l$ be a non-degenerate labeling of the vertex set $V$ of almost a ray $T$ and let $d_l$ be the corresponding ultrametric on $V$. It is shown that the ultrametric space $(V, d_l)$ is totally bounded iff this space contains an infinite totally bounded subspace. We also prove that the last property characterizes the almost rays.
\end{abstract}

\maketitle

\section{Introduction}

\cor{1}{In 2001, I.~M.~Gelfand raised the following problem: Describe, up to isometry, all finite ultrametric spaces using graph theory~\cite{Lem2001}. A representation of these spaces by monotone trees was found by V.~Gurvich and M.~Vyalyi in~\cite{GV2012DAM}. A simple geometric interpretation of Gurvich--Vyalyi representing trees was obtained in \cite{PD2014JMS}. This allows us to use the Gurvich--Vyalyi representation in various problems associated with finite ultrametric spaces. In particular, it gives a graph-theoretic interpretation of the Gomory--Hu inequality \cite{DPT2015} and characterization of finite ultrametric spaces of maximal rigidity~\cite{DPT2017FPTA}. Extremal properties of finite ultrametric spaces having strictly \(n\)-ary representing trees have been described in~\cite{DP2020pNUAA}.

The Gurvich--Vyalyi trees form a subclass of finite trees endowed with vertex labeling. J.~Gallian writes in~\cite{GalTEJoC2019} that over 200 graph labelings techniques have been studied in over 2800 paper during the past 50 years. We only note that, in almost all studies of trees with labeled vertices, it is assumed that the trees are finite. The infinite trees endowed with positive edge labelings are known as \(R\)-trees~\cite{Ber2019SMJ}. Some interrelations between finite subtrees of \(R\)-trees and Gurvich--Vyalyi trees are found in~\cite{Dov2020TaAoG}.}

Motivated by above mentioned results and some interconnections between topological and combinatorial  properties  of infinite graphs \cite{BD2006CPC, BDS2005JGT, BS2010CPC, DIESTEL2006846, DIESTEL20111423, Deistel2017, DK2004EJC, DP2017JCT, DS2011AM, DS2011TIA, DS2012DM, DMV2013AC, DP2013SM} Mehmet K\"{u}\c{c}\"{u}kaslan and the first author of the paper introduced in return the ultrametric spaces generated by vertex labelings of infinite trees \cite{DK2022AC}.

\cor{1}{Below we present connections between the combinatorial properties of infinite trees and the metric properties of ultrametric spaces generated by labeling the vertices on these trees. In particular, we show that infinite trees \(T\) that are almost rays are completely characterized by special properties of the totally bounded ultrametric spaces generated by vertex labelings on \(T\).}

In what follows, we denote by $\mathbb{R}^{+}$ the half-open interval $%
[0,\infty )$.

A \textit{metric} on a set $X$ is a function $d\colon X\times X\rightarrow
\mathbb{R}^{+}$ such that for all $x$, $y$, $z\in X$

\begin{enumerate}
\item $d(x,y)=d(y,x)$,

\item $(d(x,y)=0)\Leftrightarrow (x=y)$,

\item $d(x,y)\leq d(x,z) + d(z,y)$.
\end{enumerate}

A metric space $(X,d)$ is \emph{ultrametric} if the \emph{strong triangle
inequality}
\begin{equation*}
d(x,y)\leq \max \{d(x,z),d(z,y)\}
\end{equation*}
holds for all $x$, $y$, $z\in X$. In this case, the function $d$ is called \textit{an ultrametric} on $X$.

\begin{definition}\label{d1.3}
\cor{1}{A \emph{labeled tree} is a tree in which each vertex is assigned a single label.}
\end{definition}

Below we will consider only the non-negative real-valued vertex labelings \cor{1}{defined on arbitrary finite or infinite trees}.

Let $T=T(l)$ be a labeled tree with labeling \(l\colon V(T) \to \mathbb{R}^{+}\). Following~\cite{Dov2020TaAoG}, we define a mapping $d_l \colon V(T) \times
V(T) \to \mathbb{R}^{+}$ as
\begin{equation}  \label{e1.1}
d_l(u, v) =
\begin{cases}
0, & \text{if } u = v, \\
\max\limits_{w \in V(P)} l(w), & \text{otherwise},%
\end{cases}%
\end{equation}
where $P$ is the path joining $u$ and $v$ in $T$.

\begin{theorem}[\cite{DK2022AC}]\label{t1.4}
Let $T = T(l)$ be a labeled tree. Then the
function $d_l$ is an ultrametric on $V(T)$ if and only if the inequality
\begin{equation}\label{t1.4_eq1}
\max\{l(u), l(v)\} > 0
\end{equation}
holds for every edge $\{u, v\}$ of $T$.
\end{theorem}

We will say that a labeling $l\colon V(T) \to \mathbb{R}^{+}$ is \emph{non-degenerate} if \eqref{t1.4_eq1} is valid for every edge $\{u,v\}$ of $T$.

\begin{definition}\label{def2.4}
Let $T$ be a tree and let $A$ be a nonempty subset of $V (T)$. A subtree $H_A$ of the tree $T$ is the \emph{\cor{1}{convex hull}} of $A$ if $A \subseteq V (H_A)$ and, for every subtree $T^{*}$ of $T$, the tree $H_A$ is a subtree of $T^{*}$ whenever $A \subseteq V (T^{*})$.
\end{definition}

\begin{remark}
\cor{1}{
It should be noted that the \cor{1}{convex hull} of \(A\) exists and is unique for any nonempty \(A \subseteq V(T)\) (see Lemma~\ref{lem4.1} in the next section of the paper). The concept of the \cor{1}{convex hull} is well-known and studied by various mathematicians at least for the case of finite graphs~\cite{CMOP2005RCSiG, CJG2002CSuSGO, CWZ2002TCNoaG, Coo1971TCoTPP, FJ1986CiGaH, Gim2003SRotCNoaG, HN1981CiG, Kim2003ALBftCNoSG, Pel2013GCiG}.}
\end{remark}

The following conjecture was  formulated in \cite{DK2022AC}.

\begin{conjecture}\label{con1.7}
Let $T$ be a tree and let $(u_n)_{n\in \mathbb{N}}$ be a sequence of distinct vertices of $T$. Then, the following conditions are equivalent:

\begin{enumerate}
\item The \cor{1}{convex hull} of the range set of $(u_n)_{n\in\mathbb{N}}$ is almost a ray.

\item For every non-degenerate $l \colon V (T) \to \mathbb{R}^{+}$, the existence of a Cauchy subsequence of $(u_n)_{n\in \mathbb{N}}$ in $(V(T), d_l)$ implies that $(u_n)_{n\in \mathbb{N}}$ is also a Cauchy sequence in $(V(T), d_l)$.
\end{enumerate}
\end{conjecture}

The next result provided the starting point for Conjecture~\ref{con1.7}.

\begin{proposition}[\cite{DK2022AC}]\label{prop1.9}
Let \(R = R(l)\) be a labeled ray, $ R = (v_1, v_2, \ldots)$,
with non-degenerate labeling \(l \colon V(R) \to \mathbb{R}^{+}\). Then the sequence \((v_n)_{n \in \mathbb{N}}\) is a Cauchy sequence in \((V(R), d_l)\) if and only if this sequence contains a subsequence which is Cauchy in \((V(R), d_l)\).
\end{proposition}

The paper is organized as follows. Sections~2 and 3 contain necessary definitions and lemmas. The main results are formulated and proven in Section~4. In particular, Theorem~\ref{th4.5} shows that  Conjecture~\ref{con1.7} is valid. The totally bounded ultrametric spaces generated by vertex labelings of almost rays are characterized in Theorem~\ref{th4.2}, that is the second main result of the paper. \cor{1}{In Section~\ref{sec5} we formulate two conjectures describing the combinatorial properties of infinite trees \(T\) for which the completion of \((V(T), d_l)\) has exactly one accumulation point.}

\section{Preliminaries. Trees}

The \textit{simple graph} is a pair $(V,E)$ consisting of a nonempty set $V$
and a set $E$ whose elements are unordered pairs $\{u,v\}$ of different
points $u$, $v\in V$. In what follows we will consider the simple graphs only. For a graph $G=(V,E)$, the sets $V=V(G)$ and $E=E(G)$
are called \textit{the set of vertices} and \textit{the set of edges},
respectively.  A graph $G$ is \emph{finite} if $V(G)$ is a finite set. A graph $H$ is, by definition, a \emph{subgraph} of a graph $G$ if the inclusions $V(H)\subseteq V(G)$ and $E(H)\subseteq E(G)$ are valid. In this case, we simply write $H\subseteq G$.


\cor{1}{In what follows we will use the standard definitions of paths and cycles, see, for example, \cite[Section~1.3]{Diestel2017}. It should be noted here that we consider only those paths \(P\) that have at least two vertices.} A graph $G$ is \emph{connected} if for every two distinct $u$, $v \in V(G)$ there is a path $P[u,v] \subseteq G$ joining $u$ and $v$.


A connected graph without cycles is called a \emph{tree}.

Let $T$ be a tree containing at least two vertices and let $v \in V(T)$. In Section~4 of the paper we will denote by $T - v$ the subgraph of $T$ induced by set  $ V(T) \setminus \{v\}$. Thus
$$  V(T -v) = V(T) \setminus  \{v\}$$
and, for all $u_1, u_2 \in V(T -v)$,
$$  \left( \{ u_1,u_2\} \in E(T -v) \right) \Leftrightarrow \left( \{ u_1, u_2\} \in E(T)\right)$$
hold.

Let $v$ be a vertex of a graph $G$. The \emph{degree} of $v$ is the cardinal number ${\rm deg}_G (v)$ of the set of all  vertices of $G$ which are adjacent to $v$.

\begin{definition}\label{def2.2}
A graph $G$ is \emph{locally finite} if ${\rm deg}_G(v)$ is finite for every vertex $v \in V(G)$.
\end{definition}

An infinite graph \(R\) of the form
\begin{equation}\label{s2_eq1}
V(R) = \{v_1, v_2 \ldots, v_n, v_{n+1}, \ldots\}, \ E(R) = \bigl\{\{v_1, v_2\}, \ldots \{v_n, v_{n+1}\}, \ldots\bigr\}
\end{equation}
is called a \emph{ray} \cite{Diestel2017}. It is clear that every ray is a tree.  If $R$ is a ray  and \eqref{s2_eq1} holds, then we write  $R = (v_1, v_2, \ldots)$ for simplicity.

If $\{G_i \colon i \in I\}$ is a family of subgraphs of a graph $G$, then, by definition, the union $\bigcup_{i \in I} G_i$ is a subgraph $G^{*}$ of $G$ such that
$$
V (G^{*}) = \bigcup_{i \in I} V (G_i) \quad \textrm{and} \quad E(G^{*}) = \bigcup_{i \in I} E(G_i).
$$

\begin{definition}\label{def2.2_2}
A graph $F$ is a \emph{forest} if there is a family $ \{ T_i \colon i \in I\}$ of trees $T_i \subseteq F$ such that
$$
F= \bigcup_{i \in I} T_i
$$
and $V(T_{i_1}) \cap V(T_{i_2}) = \emptyset$ whenever $i_1 \neq i_2$. In this case we say that the trees $T_i$ are \emph{components} of the forest $F$.
\end{definition}

\begin{example}\label{example_2.4}
If $T$ is a tree containing at least two vertices, then the graph $T - v$ is a forest for every $v \in V(T)$.
\end{example}

\begin{definition}\label{def2.3}
Let $T$ be an infinite tree. Then $T$ is \emph{almost} a \emph{ray} if there are a ray $R \subseteq T$ and a finite tree $T_0 \subseteq T$ such that
$$
T= T_0 \cup R.
$$
\end{definition}

\begin{lemma}[{\cite[Proposition 8.2.1]{Diestel2017}}]  \label{lem2.4}
Every infinite connected graph has a vertex of infinite degree
or contains a ray.
\end{lemma}

\cor{1}{The following lemma is well-known for finite trees.}

\begin{lemma}\label{lem4.1}
Let $T$ be a tree, let $S \subseteq V(T)$, $\card S \geqslant 2 $, and let $H=H_S$ be the \cor{1}{convex hull} of $S$ \cor{1}{of arbitrary order}. Then the equality
\begin{equation}\label{lem4.1_eq1}
H = \bigcup_{\genfrac{}{}{0pt}{2}{v \in S}{v \neq u}} P [u, v]
\end{equation}
holds for each \(u \in S\).
\end{lemma}

\begin{proof}
Let \(u\) be an arbitrary point of \(S\). Since the tree $H$ is a subtree of $T$ and $S \subseteq V(H)$, the inclusion
$$
H \supseteq \bigcup_{\genfrac{}{}{0pt}{2}{v \in S}{v \neq u}} P [u, v]
$$
holds. Furthermore, the graph  $ \bigcup_{\genfrac{}{}{0pt}{2}{v \in S}{v \neq u}} P [u, v]$ is a connected subgraph of $T$, and, consequently, it is a subtree of $T$. It is clear that
$$
S \subseteq V \left(\bigcup_{\genfrac{}{}{0pt}{2}{v \in S}{v \neq u}} P[u,v]\right),
$$
hence we have
$$
\bigcup_{\genfrac{}{}{0pt}{2}{v \in S}{v \neq u}} P[u,v] \supseteq H
$$
by Definition~\ref{def2.4}. Equality \eqref{lem4.1_eq1} follows.
\end{proof}

\section{Preliminaries. Totally bounded ultrametric spaces}

Let $(X, d)$ be a metric space. An \emph{open ball} with a \emph{radius} $r
> 0$ and a \emph{center} $c \in X$ is the set
\begin{equation*}
B_r(c) = \{x \in X \colon d(c, x) < r\}.
\end{equation*}
In what follows, we denote by \(\mathbf{B}_X\) the set of all open balls in \((X, d)\).

\begin{definition}\label{d1.1_1}
A metric space $(X,d)$ is called \emph{totally bounded} if for every  $r \in (0, \infty)$ there is a finite set $\{B_r(c_1), \ldots, B_r(c_n)\} \subseteq \mathbf{B}_X$ such that
$$
X=\bigcup\limits_{i=1}^{n} B_r(c_i).
$$
\end{definition}

Recall also that a sequence $(x_n)_{n\in\mathbb{N}}$ of points of a metric space $(X,d)$ is a Cauchy sequence iff
\begin{equation}\label{s3_eq2}
\lim\limits_{\genfrac{}{}{0pt}{2}{n \to \infty}{m \to \infty}} d(x_n, x_m) = 0.
\end{equation}

In the next known proposition (see, for example, \cite[p.~4]{PerezGarcia2010} or \cite[Theorem 1.6, Statement (13)]{Comicheo2018}) we consider a comfortable ``ultrametric modification'' of the notion
of a Cauchy sequence.

\begin{proposition}\label{p4.5}
Let $(X, d)$  be an ultrametric space. A sequence $(x_n)_{n\in \mathbb{N}}$ of points of $X$ is a Cauchy sequence if and only if the limit relation
$$
\lim_{n\to \infty} d(x_n, x_{n+1}) = 0
$$
holds.
\end{proposition}

We will also use the following proposition.
(See, for example, Theorem~7.8.2 in \cite{Sea2007}.)

\begin{proposition}\label{pr_7} A subset $A$ of a metric space $(X, d)$  is totally bounded if and only if every infinite sequence of points of $A$ contains a Cauchy subsequence.
\end{proposition}

An important subclass of totally bounded metric spaces is the class of compact metric spaces.

\begin{definition}\label{d3.6}
Let \((X, d)\) be a metric space. A set \(A\) of \(X\) is \emph{compact} if every family \(\mathcal{F} \subseteq \mathbf{B}_X\) satisfying the inclusion
\[
A \subseteq \bigcup_{B \in \mathcal{F}} B
\]
contains a finite subfamily \(\mathcal{F}_0 \subseteq \mathcal{F}\) such that
\[
A \subseteq \bigcup_{B \in \mathcal{F}_0} B.
\]
\end{definition}

\begin{remark}\label{rem3.3}
Some new properties of totally bounded ultrametric spaces and compact ones were recently found in \cite{BDK2022TaAoG, Dov2024SUPFa, DK2022AC, DS2021TA, Ish2021pNUAA, MMWW2023TUGWD}.
\end{remark}

The next proposition gives us a generalization of Proposition~\ref{prop1.9}.

\begin{proposition}\label{prop3.3}
Let $R=(v_1, v_2, \ldots)$ be a ray and let a labeling $l\colon V(R) \to \mathbb{R}^{+}$ be non-degenerate. Let $(u_n)_{n \in \mathbb{N}}$  be a sequence  of points of $V(R)$ such that
\begin{equation}\label{pr3.3_eq1}
u_{n_1} \neq u_{n_2}
\end{equation}
whenever $n_1 \neq n_2$. Then $(u_n)_{n \in \mathbb{N}}$ is a Cauchy sequence in the ultrametric space $(V(R), d_l)$ iff $(v_m)_{m \in \mathbb{N}}$ is Cauchy in this space.
\end{proposition}

\begin{proof}
Let us  find subsequences $(v_{m_{k}})_{k \in \mathbb{N}}$ of $(v_m)_{m \in \mathbb{N}}$ and $(u_{n_k})_{k \in \mathbb{N}}$ of $(u_n)_{n \in \mathbb{N}}$ such that
\begin{equation}\label{pr3.3_eq1-2}
v_{m_k} = u_{n_k}
\end{equation}
for every $k \in \mathbb{N}$.

Let us denote by $U$ the range set of the sequence  $(u_{n})_{n \in \mathbb{N}}$,
$$
U: = \{u_n\colon n \in \mathbb{N}\}.
$$
Then, using \eqref{pr3.3_eq1}, we can define a bijection $f_u\colon U \to \mathbb{N}$ by the rule: If $x \in U$ and $n \in \mathbb{N}$  then
$$
(f_u(x) = n) \Leftrightarrow (u_n = x).
$$
Similarly, using the equality
$$
V(R)= \{v_m\colon m \in \mathbb{N} \},
$$
we define a bijective mapping $f_v\colon V(R) \to \mathbb{N}$ such that
$$
\left(f_v(x)=m \right) \Leftrightarrow (v_m = x)
$$
whenever  $x \in V(R)$ and $m \in \mathbb{N}$.

\cor{1}{Notice that \(V(R)\) inherits a well-order \(\precsim\) from \(\mathbb{N} = \{1, 2, \ldots, n, n+1, \ldots\}\) as follows. We have that
\[
x \precsim y \text{ if and only if } f_v(x) \leqslant f_v(y)
\]
for all \(x\), \(y \in V(R)\). Recall that every nonempty subset \(A\) of a well-ordered set has a unique minimal element, which we denote by \(\min_{\precsim} A\). Now, we set \(n_1 = 1\) and we define \(n_k\), \(m_k\) for \(k \geqslant 2\) by induction as follows:
\[
n_k = f_u\left(\min_{\precsim} \left\{U \setminus \{u_i \colon i \leqslant n_{k-1}\}\right\}\right)
\]
and
\[
m_k = f_v\left(\min_{\precsim} \left\{U \setminus \{u_i \colon i \leqslant n_{k-1}\}\right\}\right).
\]
Since \(U\) is infinite, the sequences \((n_k)_{k \in \mathbb{N}}\) and \((m_k)_{k \in \mathbb{N}}\) are well-defined and strictly increasing. Moreover, by construction,
\[
u_{n_k} = v_{m_k}
\]
for all \(k \in \mathbb{N}\), as desired.}

Let us prove now that $(u_n)_{n \in \mathbb{N}}$ is Cauchy iff $(v_m)_{m \in \mathbb{N}}$ is Cauchy.

Let $(u_n)_{n \in \mathbb{N}}$ be a Cauchy sequence in $(V(R), d_l)$. It was shown above that there are subsequences $(u_{n_{k}})_{k \in \mathbb{N}}$ and $(v_{m_{k}})_{k \in \mathbb{N}}$ of $(u_n)_{n \in \mathbb{N}}$ and, respectively, of $(v_m)_{m \in \mathbb{N}}$, such that \eqref{pr3.3_eq1-2} holds for every $k \in \mathbb{N}$. Since every subsequence of a Cauchy sequence is also a Cauchy sequence, $(u_{n_{k}})_{k \in \mathbb{N}}$ is Cauchy. Now using equality \eqref{pr3.3_eq1-2} we see that $(v_{m_{k}})_{k \in \mathbb{N}}$ is also a Cauchy sequence.  It implies that $(v_m)_{m \in \mathbb{N}}$ is Cauchy by Proposition~\ref{prop1.9}.

Suppose now that $(v_m)_{m \in \mathbb{N}}$ is a Cauchy sequence. We must show that $(u_n)_{n \in \mathbb{N}}$ is Cauchy. To do this it is enough, by Proposition~\ref{p4.5},  to prove the equality
\begin{equation}\label{pr3.3_eq14}
\lim_{n \to \infty} d_l (u_n, u_{n+1})=0.
\end{equation}

Let us define a number $m_n \in \mathbb{N}$ as
$$
m_n := f^{*}_u(u_n)
$$
for each $n \in \mathbb{N}$, \cor{1}{where \(f^{*}\) is the mapping
\[
U \xrightarrow{\mathrm{id}} V(R) \xrightarrow{f_v} \mathbb{N}
\]
such that $\mathrm{id} (x) = x$ for every $x \in U$}. Then equality \eqref{pr3.3_eq14} holds iff
\begin{equation}\label{pr3.3_eq15}
\lim_{n \to \infty} d_l (v_{m_n}, v_{m_{n+1}}) = 0.
\end{equation}
To prove the last equality we note that
\begin{equation}\label{pr3.3_eq16}
\lim_{n \to \infty} m_n = \infty
\end{equation}
because $f^{*}_u$ is injective and  \eqref{pr3.3_eq1} holds whenever $n_1 = n_2$. Limit relation \eqref{pr3.3_eq16} implies
\begin{equation}\label{pr3.3_eq17}
\lim_{n+1 \to \infty} m_{n+1}= \infty.
\end{equation}
Now using the definition of Cauchy sequences (see \eqref{s3_eq2}) we obtain
$$
\lim_{\genfrac{}{}{0pt}{2}{p \to \infty}{q \to \infty}} d_l (v_p, v_q) =0.
$$
The last equality with $p=m_n$ and $q=m_{n+1}$, \eqref{pr3.3_eq17} and \eqref{pr3.3_eq16} imply \eqref{pr3.3_eq15}. The proof is  completed.
\end{proof}

\cor{1}{
\begin{remark}
The first part of the above proof was proposed by one of the reviewers of the paper.
\end{remark}
}

We conclude this section by following

\begin{lemma}\label{lem3.5}
Let \(T = T(l)\) be a labeled tree with non-degenerate labeling. If the ultrametric space \((V(T), d_l)\) is totally bounded, then the set \(V(T)\) is at most countable.
\end{lemma}

For the proof see Corollary~6 in \cite{DK2022AC}.

\section{Main results}

The following theorem shows that Conjecture~\ref{con1.7} is valid.

\begin{theorem}\label{th4.5}
Let $T$ be a tree and suppose that $(u_n)_{n\in \mathbb{N}}$ is a sequence of distinct vertices in $T$.
Let $H$ be the \cor{1}{convex hull} of $\{u_n \colon n \in \mathbb{N} \}$. The following conditions are equivalent:
\begin{itemize}
\item[$(i)$] $H$ is almost a ray.
\item[$(ii)$]  For every non-degenerate $l\colon V(T) \to \mathbb{R}^{+}$, the existence of a Cauchy subsequence of the sequence $(u_n)_{n \in \mathbb{N}}$ implies that $(u_n)_{n \in \mathbb{N}}$ is also Cauchy in $(V(T), d_l)$.
\end{itemize}
\end{theorem}

\begin{proof}
$(i) \Longrightarrow (ii)$. Let $H$ be the union of a ray $R= (v_1, v_2, \ldots)$ with a finite tree $T^0$,
$$
H=R \cup T^0,
$$
and let $l \colon V (T) \to \mathbb{R}^+$ be a non-degenerate labeling. Suppose that a subsequence $(u_{n_k})_{k\in \mathbb{N}}$ of the sequence  $(u_{n})_{n\in \mathbb{N}}$ is a Cauchy sequence in the space $(V(T), d_l)$, where the ultrametric $d_l$ is defined by \eqref{e1.1}. We must show that $(u_n)_{n\in \mathbb{N}}$ is also Cauchy in $(V(T), d_l)$.
To do this, we will consider the restrictions $l_H$ and $l_R$ of the labeling $l$ on the sets $V(H)$ and $V(R)$ respectively.
Using Lemma~\ref{lem4.1} and Theorem~\ref{t1.4} we see that the function $d_{l_H} \colon V(H) \times V(H) \to \mathbb{R}^+ $ defined as in \eqref{e1.1} is an ultrametric on $V(H)$ and, that the equality
\begin{equation}\label{th4.5_eq1}
d_l (x,y) = d_{l_H} (x,y)
\end{equation}
holds for all $x,y \in V(H)$. Thus it suffices to show that $(u_n)_{n\in \mathbb{N}}$ is a Cauchy sequence in $(V(H), d_{l_H})$. Using \eqref{th4.5_eq1} we see that $(u_{n_k})_{k\in \mathbb{N}}$ is Cauchy in $(V(H), d_{l_H})$. Since $T^0$ is finite, we have
$$
u_n \in V(R) \quad \textrm{and}  \quad u_{n_k} \in V(R)
$$
if $n$ and $k$ are large enough. Consequently  there is $r_0 \in \mathbb{N}$ such that the range sets of the tails $(u_{n + r_0})_{n \in \mathbb{N}}$ and $(u_{n_{k + r_0}})_{k \in \mathbb{N}}$ are subsets of $V(R)$.

Using \eqref{e1.1} and Definition \ref{def2.3} we obtain the equality
\begin{equation}\label{th4.5_eq2}
d_{l_H} (x,y) = d_{l_R} (x,y)
\end{equation}
for all $x, y \in V(R)$.
Consequently  $(u_n)_{n\in \mathbb{N}}$ is Cauchy in $(V(H), d_{l_H})$ if and only if $(u_{n+r_0})_{n\in \mathbb{N}}$ is Cauchy in $(V(R), d_{l_R})$. Furthermore, using \eqref{th4.5_eq2} we see that $(u_{n_{k+r_0}})_{k\in \mathbb{N}}$ is a Cauchy sequence in $(V(R), d_{l_R})$. Hence, by Proposition~\ref{prop3.3}, the sequence $(v_n)_{n \in \mathbb{N}}$ is a Cauchy sequence in $(V(R), d_l)$. The last statement and Proposition~\ref{prop3.3} imply that $(u_{n+r_0})_{n\in \mathbb{N}}$ is Cauchy in $(V(R), d_{l_R})$.

It follows that $(u_n)_{n\in \mathbb{N}}$ is Cauchy as required.

$(ii) \Longrightarrow (i)$. Supposing that $H$ is not a union of a ray with some finite tree, we will show that $(ii)$ does not hold.
Using Lemma~\ref{lem2.4} we see that the following cases are possible.

\textbf{Case 1}. $H$ has a vertex $u$ of infinite degree.

\textbf{Case 2}. $H$ has a ray $R= (v_1, v_2, \dots)$.

\cor{1}{\textbf{Case 2.1}}. $H \supseteq R$ and the set
\begin{equation}\label{th4.4_case2.0}
S:= \{u_n \colon n \in \mathbb{N}\} \setminus V(R)
\end{equation}
is finite.

\cor{1}{\textbf{Case 2.2}}. $H \supseteq R$ and $S$ is infinite.

\cor{1}{\textbf{Case 2.2.1}. $H \supseteq R$, $S$ and and the set
\begin{equation}\label{th4.4_case2.1.1}
I:= \{u_n\colon n \in \mathbb{N}\} \cap V(R)
\end{equation}
are infinite.

\textbf{Case 2.2.2}. $H \supseteq R$, $S$ is infinite and \(I\) is finite.}

Thus, the structure of the proof of the validity of $(ii) \Longrightarrow (i)$ can be represented by Figure~1.

\begin{figure}[ht]
\begin{center}
\def\mh{0.7cm} 
\def\mw{2cm} 
\begin{tikzpicture}[x=0.65pt, y=0.65pt,yscale=-1,xscale=1,black,
everybox/.style = {
rectangle,
thick,
draw=black,
align=center,
inner sep=2pt,
anchor=north west,
minimum height = \mh},
box/.style = {everybox, minimum width =\mw},
]

\node (root) [box] at (0, 0) {\textit{H} is not \cor{1}{almost} a ray};

\node (case1)  [box, below left = .5cm and -.2 cm of root] {Case 1};
\node (case2)  [box, below right = .5cm and -.2 cm of root] {Case 2};

\node (case20)  [box, below left = .5cm and -.2 cm of case2] {Case 2.1};
\node (case21)  [box, below right = .5cm and -.2 cm of case2] {Case 2.2};

\node (case210)  [box, below left = .5cm and -.5 cm of case21] {Case 2.2.1};
\node (case211)  [box, below right = .5cm and -.5 cm of case21] {Case 2.2.2};

\draw [->] (root) -- (case1);
\draw [->] (root) -- (case2);

\draw [->] (case2) -- (case20);
\draw [->] (case2) -- (case21);

\draw [->] (case21) -- (case210);
\draw [->] (case21) -- (case211);
\end{tikzpicture}
\end{center}
\center\textrm{Figure 1. \\
The structure of the proof of the second part of Theorem~4.1.}
\label{fig1}
\end{figure}

\textbf{Case 1.} $H$ has a vertex $u$ of infinite degree.  Lemma~\ref{lem4.1} implies that the vertex set $V(H)$ is countable. Since $H$ is a tree with countable $V(H)$, the graph $H - u$ is a forest consisting of a countable set $\{T_1, T_2, \ldots \}$ of components.
Observe that each component $T_i$ has a vertex that is an element of $(u_n)_{n \in \mathbb{N}}$. Let us construct a non-degenerate labeling $l\colon V(T) \to \mathbb{R}^{+}$ such that
\begin{equation}\label{th4.4_eq2}
l(v) =\left\{
\begin{array}{ll}
0, & \quad \hbox{if}  \quad v = u,\\
\frac{1}{n}, & \quad \hbox{if}  \quad v \in V(T_{2n}) \ \textrm{for some } n \in \mathbb{N}, \\
1, & \quad \hbox{otherwise.}
\end{array}
\right.
\end{equation}
Let us define a sequence $(n_k)_{k \in \mathbb{N}}$, $n_1 < n_2 < \ldots$, recursively by setting
$$
n_1 := 1 \quad \textrm{and} \quad n_k := \min \{i > n_{k-1} \colon l(u_i) \leqslant 1/k\} \quad \textrm{for} \quad k \geqslant 2.
$$
By Proposition~\ref{p4.5} the sequence $(u_{n_k})_{k\in \mathbb{N}}$ is a Cauchy subsequence of $(u_{n})_{n\in \mathbb{N}}$, since the inequality
$$
d_l (u_{n_k}, u_{n_{k+1}}) \leqslant \frac{1}{k}
$$
holds for every $k \in \mathbb{N}$.  It remains to note that $(u_n)_{n \in \mathbb{N}}$ cannot be a Cauchy sequence in $(V(T), d_l)$, since, by \eqref{th4.4_eq2}, there are infinitely many $n \in \mathbb{N}$ such that
$$
l(u_n) =1
$$
holds. This proves that $(ii)$ does not hold.

\cor{1}{\textbf{Case 2.1.}}  If $S$ is finite, then the \cor{1}{convex hull}  of the set $S \cup \{v_1\}$ is a finite subtree $H^1$ of the tree $H$. Since
$$
v_1 \in V(R) \cap V(H^{1}),
$$
the graph $H^1 \cup R$ is a connected subgraph of the tree $H$ and, consequently it is a subtree of $H$. The set $\{u_n \colon n \in \mathbb{N} \} $ is a subset of $V (R) \cup V(H^{1})$. Hence
\begin{equation}\label{th4.4_eq5}
H^{1} \cup R \supseteq H
\end{equation}
holds by Definition~\ref{def2.4}. Now using \eqref{th4.4_eq5} and
$$
H \supseteq H^1, \quad H \supseteq R
$$
we obtain the equality
$$
H= H^1 \cup R.
$$
Thus, $H$ is almost a ray contrary to the supposition.

\cor{1}{{\bf Case 2.2.1.}}  Let us construct a non-degenerate  labeling $l\colon V (T) \to  \mathbb{R}^+$ such that
\begin{equation}\label{th4.4_eq6}
l(v) =
\left\{
\begin{array}{ll}
\frac{1}{n}, & \quad \hbox{if} \quad v= v_n \quad \textrm{for some } \quad n \in \mathbb{N},\\
1, &  \quad \hbox{otherwise}.
\end{array}
\right.
\end{equation}
Since $S$ is infinite, there are infinitely many $k \in \mathbb{N}$ such that
$$
l(u_k) = 1,
$$
so $(u_n)_{n \in \mathbb{N}}$ is not Cauchy in the space $(V(T), d_l)$. Using \eqref{th4.4_eq6} it is easy to see that the sequence $(v_n)_{n \in \mathbb{N}}$ is a Cauchy sequence in $(V(T), d_l)$. Since $I$ is infinite, there is an infinite subsequence $(u_{n_k})_{k \in \mathbb{N}}$  of the sequence $(u_n)_{n \in \mathbb{N}}$ satisfying the inclusion
$$
\{u_{n_k} \colon k \in \mathbb{N} \} \subseteq V(R).
$$

We claim that $(u_{n_k})_{k \in \mathbb{N}}$ is Cauchy in $(V(T), d_l)$.

Let $l_R$ be a restriction of  the labeling $l\colon V(T) \to \mathbb{R}^+ $ on the set $V(R)$, where $l$ is determined by formula \eqref{th4.4_eq6}. Then \eqref{e1.1} implies the equality
$$
d_l (x,y) = d_{l_R} (x,y)
$$
for all $x,y \in R$. Hence it suffices to show that $(u_{n_k})_{k \in \mathbb{N}}$ is a Cauchy sequence in $(V(R), d_{l_R})$ which follows directly from Proposition~\ref{prop3.3}.

\cor{1}{\textbf{Case 2.2.2.}} Since $I$ is finite, there is $n_0 \in \mathbb{N}$ such that
$$
\{ u_n \colon n \in \mathbb{N}\} \cap \{ v_{n+n_0}\colon n \in \mathbb{N}\}  = \emptyset.
$$
Thus there is a ray $R^{*} = (v_1^{*}, v_2^{*}, \dots)$ such that $R^{*} \subseteq H$ and
\begin{equation}\label{th4.4_eq7}
\{ u_n \colon n \in \mathbb{N}\} \cap V(R^{*}) = \emptyset.
\end{equation}
See Figure~2 below  for an example of a sequence $(u_n)_{n \in \mathbb{N}}$ satisfying \eqref{th4.4_eq7}.

\begin{figure}[ht]
\begin{center}
\begin{tikzpicture}[
level 1/.style={level distance=\levdist,sibling distance=24mm},
level 2/.style={level distance=\levdist,sibling distance=12mm},
level 3/.style={level distance=\levdist,sibling distance=6mm},
solid node/.style={circle,draw,inner sep=1.5,fill=black},
micro node/.style={circle,draw,inner sep=0.5,fill=black}]

\def\dx{1.3cm}
\draw (\dx, 0.5) node [solid node, label= below:\(u_1\)] {} -- (2*\dx, 0) node [solid node, label= below:\(v_1^{*}\)] {};
\draw (\dx, -0.5) node [solid node, label= below:\(u_2\)] {} -- (2*\dx, 0) node [solid node, label= below:\(v_1^{*}\)] {};
\draw [solid node] (2*\dx, 0) -- (3*\dx, 0) node [solid node, label= below:\(v_2^{*}\)] {};
\draw (3*\dx, 0) -- (4*\dx, 0) node [solid node, label= below:\(v_3^{*}\)] {} -- (5*\dx, 0) node [solid node] {};

\draw (\dx+2, 1.5) node  [anchor=north west][inner sep=0.75pt]   [align=left] {$H$};

\draw (3*\dx, 0) -- (3*\dx, \dx) node [solid node, label=left:\(u_3\)] {};
\draw (4*\dx, 0) -- (4*\dx, \dx) node [solid node, label=left:\(u_4\)] {};

\draw [dashed] (5*\dx, 0) -- (6*\dx, 0) node [solid node] {};
\draw (6*\dx, 0) -- (7*\dx, 0) node [solid node, label= below:\(v_n^{*}\)] {} -- (8*\dx, 0) node [solid node, label= below:\(v_{n+1}^{*}\)] {};
\draw (7*\dx, 0) -- (7*\dx, \dx) node [solid node, label= left:\(u_{n+1}\)] {};
\draw (8*\dx, 0) -- (8*\dx, \dx) node [solid node, label= left:\(u_{n+2}\)] {};
\draw [dashed] (8*\dx, 0) -- (9*\dx, 0);
\end{tikzpicture}
\end{center}
\center\textrm{Figure 2. \\
The tree $H$ is the \cor{1}{convex hull} of the set $\{ u_n \colon n \in \mathbb{N}\}$.}
\label{fig3}
\end{figure}

Let $ ({\rm deg}_H (v_n^{*}))_{n \in \mathbb{N}}$ be the sequence of degrees of vertices of $R^{*}$. We claim that there is an infinite subsequence $ ({\rm deg}_H (v_{n_k}^{*}))_{k \in \mathbb{N}}$ of the sequence $ ({\rm deg}_H (v_n^{*}))_{n \in \mathbb{N}}$ such that $n_1 \geqslant 2$ and
\begin{equation}\label{th4.4_eq8}
{\rm deg}_H (v_{n_k}^{*}) \geqslant  3
\end{equation}
for every $k \in \mathbb{N}$.

Suppose to the contrary that
\begin{equation}\label{th4.4_eq8_2}
{\rm deg}_H (v_n^{*}) \leqslant 2
\end{equation}
holds whenever $n \geqslant n_1 >2$, where $n_1 \in \mathbb{N}$ is a given number. Since the equality
$$
{\rm deg}_{R^{*}} (v_n^{*}) = 2
$$
holds for every natural $n \geqslant 2$, inequality \eqref{th4.4_eq8_2} and $R^{*} \subseteq H $ imply that
$$
{\rm deg}_{H} (v_n^{*}) = 2
$$
holds for every $n \geqslant n_1$. Consequently the graph $H - v_{n_1}^{*}$ is a forest consisting a tree $T_1$ and the ray $R_1^{*} = (v_{n_1 +1}^{*}, v^{*}_{n_1 + 2}, \ldots),$
\begin{equation}\label{th4.4_eq9}
H - v_{n_1}^{*} = T_1 \cup R_1^{*}.
\end{equation}

Since $V(R_1^{*}) \subseteq V(R^{*})$ holds, equality \eqref{th4.4_eq7} implies
\begin{equation}\label{th4.4_eq10}
\{u_n \colon n \in \mathbb{N} \} \cap V(R^{*}) = \emptyset.
\end{equation}
Moreover, using \eqref{th4.4_eq7} we obtain
\begin{equation}\label{th4.4_eq11}
v_{n_1}^{*} \notin \{u_n \colon n \in \mathbb{N} \}.
\end{equation}
It follows from \eqref{th4.4_eq9}--\eqref{th4.4_eq11} that
$$
\{u_n \colon n \in \mathbb{N} \} \subseteq V(T_1).
$$
Hence the \cor{1}{convex hull} $H$ satisfies the inclusion
$$
H \subseteq T_1
$$
by Definition~\ref{def2.4}, which contradicts equality~\eqref{th4.4_eq9}. Thus inequality \eqref{th4.4_eq8} holds for every $k \in \mathbb{N}$.

The forest $H - v_{n_k}^{*}$ contains exactly ${\rm deg}_{H} (v^{*}_{n_k})$ components. By \eqref{th4.4_eq8}, among these components there are trees $T^{+}_{n_k}$ and $T^{-}_{n_k}$ such that
\begin{equation}\label{th4.4_eq12}
v_{n_k +1}^{*}\in T^{+}_{n_k} \quad \textrm{ and } \quad v^{*}_{n_k -1}\in T^{-}_{n_k}.
\end{equation}
Using condition \eqref{th4.4_eq8} we see that the set
\begin{equation}\label{th4.4_eq13}
V_{n_k} := V(H - v_{n_k}^{*}) \setminus \left( V(T_{n_k}^{+}) \cup V(T_{n_k}^{-})\right)
\end{equation}
is nonempty,
\begin{equation}\label{th4.4_eq14}
V_{n_k} \neq \emptyset.
\end{equation}
Let us construct a non-degenerate labeling $l\colon V(T) \to \mathbb{R}^{+}$ such that
\begin{equation}\label{th4.4_eq15}
l(v) =\left\{
\begin{array}{ll}
\frac{1}{n}, & \quad v= v^{*}_n \quad \hbox{for some}  \quad n \in \mathbb{N},\\
\ \\
\frac{1}{n_k}, & \quad \hbox{if}  \quad v \in V_{n_k} \ \textrm{and } \ k \ \textrm{is even}, \\ \ \\
1, & \quad \hbox{otherwise.}
\end{array}
\right.
\end{equation}
Since
$$
V_{n_{k_1}} \cap  V(R^{*})  = \emptyset
$$
holds for every $k \in \mathbb{N}$ and, in addition,
\begin{equation}\label{th4.4_eq16}
V_{n_{k_1}} \cap V_{n_{k_2}}  = \emptyset
\end{equation}
whenever $k_1 \neq k_2$, the labeling $l$ is correctly defined.

It follows from \eqref{th4.4_eq13}--\eqref{th4.4_eq14} that the set
$$
W_k := \{u_n \colon n \in \mathbb{N}\} \cap V_{n_k}
$$
is nonempty for every $k \in \mathbb{N}$. In particular if $k$ is an odd number, then \eqref{th4.4_eq15} implies
\begin{equation}\label{th4.4_eq17}
l(v)=1
\end{equation}
for every $v \in W_k$. Now using \eqref{th4.4_eq16} and \eqref{th4.4_eq17} it is easy to prove that $(u_n)_{n \in \mathbb{N}}$ is not Cauchy in the ultrametric space $(V(T), d_l)$. Let us consider now some distinct even numbers $k_1$ and $k_2$. Then using \eqref{th4.4_eq15} and \eqref{th4.4_eq16} we obtain
$$
d_l (w_1, w_2) = d_l (v^{*}_{n_{k_1}}, v^{*}_{n_{k_2}}) = \max \left \{ \frac{1}{n_{k_1}}, \frac{1}{n_{k_2}}\right \}.
$$
Consequently the sequence $(u_n)_{n\in \mathbb{N}}$ contains an infinite Cauchy subsequence. Thus $(ii)$ does not hold.

The proof is completed.
\end{proof}

Analyzing the second part of the proof of Theorem~\ref{th4.5} we obtain the following.

\begin{corollary}\label{cor4.2}
Let $T$ be a tree, $(u_n)_{n \in \mathbb{N}}$ be a sequence of distinct vertices of $T$, and let $H$ be the \cor{1}{convex hull} of $\{ u_n\colon n \in \mathbb{N} \}$. If $H$ is not almost a ray, then there is a non-degenerate labeling $l\colon V(T) \to \mathbb{R}^+$ such that the sequence $(u_n)_{n \in \mathbb{N}}$ has a Cauchy subsequence but the set
$$
\{ v \in V(E)\colon l(v) =1 \}
$$
is infinite.
\end{corollary}

The next theorem is the second main result of the paper.

\begin{theorem}\label{th4.2}
Let $T$ be an infinite tree. Then the following conditions are equivalent:
\begin{itemize}
\item[$(i)$] $T$ is almost a ray.
\item[$(ii)$]  For every non-degenerate labeling $l\colon V(T) \to \mathbb{R}^+$ the space $(V(T), d_l)$ is totally bounded iff there is an infinite totally bounded  $A \subseteq V(T)$.
\end{itemize}
\end{theorem}

\begin{proof}
$(i) \Longrightarrow (ii)$. Let $T$ be almost a ray, $l\colon V(T) \to \mathbb{R}^{+}$ be non-degenerate, and let $A$ be an infinite totally bounded subset of the space $(V(T), d_l)$. We must show that $(V(T), d_l)$ is totally bounded.

Since $T$ is almost a ray, the set $V(T)$ is countable and, consequently, the vertices of $T$ can be numbered in a sequence $(u_n)_{n \in \mathbb{N}}$ such that
$$
u_{n_1} \neq u_{n_2}
$$
whenever $n_1 \neq n_2$. The equality
\begin{equation}\label{th4.2_eq1}
V(T) = \{ u_n \colon n \in \mathbb{N}\}
\end{equation}
and Definition~\ref{def2.4} imply that the tree $T$ is the \cor{1}{convex hull} of the set $\{u_n \colon n \in \mathbb{N}\}$. Let us prove now that $(u_n)_{n \in \mathbb{N}}$ is a Cauchy sequence in $(V(T), d_l)$.
According to Theorem~\ref{th4.5}, for this it is sufficient to find a Cauchy subsequence of the sequence $(u_n)_{n \in \mathbb{N}}$.

Since $A$ is an infinite subset of $V(T)$, equality \eqref{th4.2_eq1} implies that $(u_n)_{n \in \mathbb{N}}$ contains a subsequence $(u_{n_k})_{k \in \mathbb{N}}$ such that $u_{n_k} \in A$ for every $k \in \mathbb{N}$. Let us write for simplicity
$$
u_k^{*} = u_{n_k}, \quad k \in \mathbb{N}.
$$
Since $A$ is a totally bounded subset of $(V(T), d_l)$, the sequence $(u_k^{*})_{k \in \mathbb{N}}$ contains a Cauchy subsequence $(u_{k_l}^{*})_{l \in \mathbb{N}}$ by Proposition~\ref{pr_7}. The sequence $(u_{k_l}^{*})_{l \in \mathbb{N}}$ also is a subsequence of $(u_n)_{n \in \mathbb{N}}$. Thus $(u_{n})_{n \in \mathbb{N}}$ is a Cauchy sequence in $(V(T), d_l)$ by Theorem~\ref{th4.5}.

Let us consider now an arbitrary infinite sequence $(v_n)_{n \in \mathbb{N}}$ of points of $(V(T), d_l)$. To prove that $(V(T), d_l)$ is totally bounded it is sufficient, by Proposition~\ref{pr_7}, to find a Cauchy subsequence of $(v_n)_{n \in \mathbb{N}}$. If the range set of $(v_n)_{n \in \mathbb{N}}$ is finite, then there is $(v_{n_k})_{k \in \mathbb{N}}$ such that $v_{n_{k_1}} = v_{n_{k_2}}$ for all $k \in
\mathbb{N}$. It is clear that  $(v_{n_k})_{k \in \mathbb{N}}$ is Cauchy. For the case of an infinite range set of $(v_n)_{n \in \mathbb{N}}$, the sequences $(u_n)_{n \in \mathbb{N}}$ and $(v_n)_{n \in \mathbb{N}}$ have a common subsequence. It was proved above that $(u_n)_{n \in \mathbb{N}}$  is a Cauchy sequence. Consequently that  common subsequence also is  Cauchy as required.

$(ii) \Longrightarrow (i)$.  Let condition $(ii)$ hold. We must show that $T$ is almost a ray. \cor{1}{Using Lemma~\ref{lem2.4}, we can find an infinite subtree \(T_0\) of the tree \(T\) and a non-degenerate labeling \(l_0 \colon V(T_0) \to \mathbb{R}^{+}\) such that \((V(T_0), d_{l_0})\) is a totally bounded ultrametric space. Let us define \(l^* \colon V(T) \to \mathbb{R}^{+}\) as
\[
l^*(v) = \begin{cases}
l_0(v), & \text{if } v \in V(T_0),\\
1, & \text{otherwise}.
\end{cases}
\]
Then the labeling \(l^*\) is also non-degenerate. Moreover, since \(T_0\) is a subtree of \(T\), equality~\eqref{e1.1} implies the equality
\[
d_{l_0}(u, v) = d_{l^*}(u, v)
\]
for all \(u\), \(v \in V(T_0)\). Consequently, \(V(T_0)\) is an infinite totally bounded subset of the ultrametric space \((V(T), d_{l^*})\). Now, using~\((ii)\) with \(l = l^*\) and \(A = V(T_0)\), we see that \((V(T), d_{l^*})\) is totally bounded.}

Since $T$ is an infinite tree, the set $V(T)$ is countable by Lemma~\ref{lem3.5}. Consequently, there is a sequence $(u_n)_{n \in \mathbb{N}}$ of distinct points of $V(T)$ such that \eqref{th4.2_eq1} holds. It is clear that $T$ is the \cor{1}{convex hull} of $\{u_n \colon n \in \mathbb{N} \}$.

Suppose that $T$ is not almost a ray. Then, by Corollary~\ref{cor4.2}, there is a non-degenerate labeling $l\colon V(T) \to \mathbb{R}^{+}$ such that  the  sequence $(u_n)_{n \in \mathbb{N}}$ has a Cauchy subsequence  $(u_{n_k})_{k \in \mathbb{N}}$ but the set
\begin{equation}\label{th4.2_eq2}
D=\{ v\in V(T) \colon l(v)=1\}
\end{equation}
is infinite. Let us denote by $A$ the range set of $(u_{n_k})_{k \in \mathbb{N}}$ and consider an arbitrary infinite sequence $(v_{n})_{n \in \mathbb{N}}$ of distinct points of $A$. Arguing in the same way as in the first part of the proof, we can find a Cauchy  subsequence of the sequence $(v_{n})_{n \in \mathbb{N}}$. Hence $A$ is a totally bounded subset of $(V(T), d_l)$ by Proposition~\ref{pr_7}. Since $A$ is infinite, condition $(ii)$ implies that $(V(T), d_l)$ also is totally bounded. Consequently the set $D$, defined by \eqref{th4.2_eq2}, must be totally bounded as a subset of a totally bounded metric space. To complete the proof we note that $D$ cannot be totally bounded, because $D$ is infinite and, moreover, \eqref{e1.1} and \eqref{th4.2_eq2} imply the inequality
$$
d_l (x,y) \geqslant 1
$$
for all distinct $x,y \in D$. Thus $T$ is almost a ray, as required.
\end{proof}

Recall that a point $p$ of a metric space $(X,d)$ is \emph{isolated} if there is $\varepsilon \in (0, \infty)$ such that
$$
d(x, p) > \varepsilon
$$
for every $x \in X \setminus \{p\}$.

\begin{corollary}\label{cor4.4}
Let $T$ be an infinite tree. Suppose that for every non-degenerate labeling $l\colon V(T) \to \mathbb{R}^+$ the space $(V(T), d_l)$ is totally bounded iff there is an infinite totally bounded  $A \subseteq V(T)$. Then the ultrametric space $(V(T), d_l)$ is not compact for every non-degenerate labeling $l\colon V(T) \to \mathbb{R}^{+}$.
\end{corollary}

\begin{proof}
Theorem~\ref{th4.2} implies that $T$ is almost a ray. Hence $T$ is locally finite by Definition~\ref{def2.2}. Let $l\colon V(T) \to \mathbb{R}^{+}$ be non-degenerate. Since $T$ is locally finite, Theorem~\ref{t1.4} implies that every point of $(V(T), d_l)$ is isolated. Thus $(V(T), d_l)$ is not compact by Definition~\ref{d3.6}.
\end{proof}

\cor{1}{
\section{Two conjectures}\label{sec5}

Theorems~\ref{th4.5} and \ref{th4.2} imply that, for every almost ray \(T\) and each non-degenerate labeling \(l \colon V(T) \to \mathbb{R}^{+}\), the completion of the ultrametric space \((V(T), d_l)\) has at most one accumulation point.

\begin{conjecture}\label{con5.1}
Let \(T\) be an infinite, locally finite tree. Then the following statements are equivalent:
\begin{enumerate}
\item[\((i)\)] The completion of ultrametric space \((V(T), d_l)\) has exactly one accumulation point for every vertex labeling \(l \colon V(T) \to \mathbb{R}^{+}\) generating a totally bounded ultrametric \(d_l\).
\item[\((ii)\)] The tree \(T\) is a union of a ray \(R\) and forest \(F\), i.e., \(T = R \cup F\), such that all components of \(F\) are finite trees.
\end{enumerate}
\end{conjecture}

It should be noted here that a tree \(T\) is locally finite if and only if there is a vertex labeling \(l \colon V(T) \to \mathbb{R}^{+}\) generating a discrete totally bounded ultrametric \(d_l\) on \(V(T)\), see Corollary~8 in \cite{DK2022AC}.

Our next goal is to give a modification of Conjecture~\ref{con5.1} for the case of compact \((V(T), d_l)\). The completion of every totally bounded metric space is compact but it was shown in Theorem~7 of \cite{DK2022AC} that the existence of a labeling \(l \colon V(T) \to \mathbb{R}^{+}\) generating a compact ultrametric \(d_l\) implies that the tree \(T\) is rayless.

\begin{conjecture}\label{con5.2}
Let \(T\) be a rayless countable tree. Then the following statements are equivalent:
\begin{enumerate}
\item[\((i)\)] An ultrametric space \((V(T), d_l)\) has exactly one accumulation point for every vertex labeling \(l \colon V(T) \to \mathbb{R}^{+}\) generating a compact \(d_l\).
\item[\((ii)\)] The tree \(T\) has exactly one vertex of infinite degree.
\end{enumerate}
\end{conjecture}

It seems to be interesting to get a purely metric characterization of compact ultrametric spaces generated by labeled trees.

}

\section*{Conflict of interest statement}

\cor{1}{The research was conducted in the absence of any commercial or financial relationships that could be construed as a potential conflict of interest.}

\section*{Funding}

\cor{1}{The first author was supported by grant 359772 of the Academy of Finland.}

\section*{Acknowledgement}
The authors express their gratitude to anonymous referees whose non-trivial remarks strongly helped us in the preparing of the final version of this paper.


\end{document}